\documentclass[10pt,draft]{amsart}
\usepackage{amsmath,amssymb}
\usepackage[all]{xy}
\usepackage[nobysame]{amsrefs}

\title{Algebraic structure representations for lattices}
\author{Martha L.H. Kilpack}
\address{Department of Mathematics, Brigham Young University, Provo,
  UT 84602 USA}
\email{mlhkilpack@mathematics.byu.edu}
\author{Ryan Kurth-Oliveira} 
\address{Department of Mathematics, Brigham Young University, Provo,
  UT 84602 USA}
\email{rkurth4@yahoo.com}
\author{Madeline E. May} 
\address{Department of Mathematics, Brigham Young University, Provo,
  UT 84602 USA}
\email{madelinemay1996@gmail.com}


\usepackage{graphicx}
\usepackage{tikz}
\usepackage{amsthm,amsmath}
\usepackage{amssymb}
\usepackage{verbatim}
\usepackage{MnSymbol}
\usepackage{multicol}

\newtheorem{proposition}{Proposition}
\newtheorem{theorem}{Theorem}
\newtheorem{definition}{Definition}

\begin{document}

\maketitle

\begin{abstract}
 For an arbitrary group, the subgroups form a lattice with order determined by set inclusion.  Not every lattice is isomorphic to the subgroup lattice for a group. However, Birkhoff and Frink proved that any compactly generated lattice is isomorphic to a subalgebra lattice for some algebraic structure. An algebraic structure is a set $A$ with operations from $A^n$ to $A$ where $n$ is a non-negative integer.  Although the proof by Birkhoff and Frink is constructive, many of the operations described are not needed for an algebraic structure to represent a given lattice. In this paper we utilize concepts in the proof by Birkhoff and Frink to describe and count functions that are used to create algebraic structure representations for certain finite lattice types.

 \end{abstract}

\section{Introduction} We put items into an order everyday: lines at fast food restaurants, tasks to complete, favorite food items ...  An ordering is not always linear in nature. Two examples of non-linear orderings are the natural numbers with $x$ being ``less than" $y$ when $x$ is a divisor of $y$, and sets under the subset ordering. A special type of ordering is that of a lattice.  All linear orderings are lattices and the orderings of divisor and subset are also examples of lattices. 

One special lattice is that of subgroups.  In the mathematical group, which is a set with a binary operation, subsets that are closed under the binary operation are subgroups.  These subgroups form a lattice under the subset ordering.  The subgroup ordering for a group $G$, called a subgroup lattice, allows one to discover information about the original group \cite{subgroup}.  

Although subgroup lattices are an important type of lattice for study, not all lattices are isomorphic to a subgroup lattice. Birkhoff and Frink showed, given compactly generated lattice $L$, an algebraic structure (set with operations) $\textbf{A}$ can be found where $L$ is isomorphic to the subalgebra lattice of $\textbf{A}$ \cite{birk_frink}. We call $\textbf{A}$ an algebraic representation of $L$.  Since all finite lattices are compact, all finite lattices have an algebraic representation.    

Birkhoff and Frink's proof for algebraic representations is constructive.  This means for a finite lattice, one can produce an algebraic representation. However, this construction creates an algebraic structure with numerous unnecessary functions.  For example, following the proof exactly, a lattice of four elements would be represented by an algebraic structure with 4 elements and 50 functions.  It would be better to find an algebraic representation without the unnecessary functions.  

In this paper we will look at different families of lattices and, utilizing ideas from the Birkhoff and Frink theorem, count the number of functions in a representation for the family of lattices.

\pagebreak

\section{Preliminaries and Notation}

In this section we look at useful definitions and notations. We utilize definitions and notations from Burris and Sankappanavar's book \textit{A Course in Universal Algebra} chapter two \cite{univ_alg}*{25--37}. 

\begin{definition}
An algebraic structure is \textbf{A} $= (A, F)$ for a set $A$ and functions $f \in F$ $f: A^n \to A$ where $n \in \mathbb{N} \cup \{0\}$.
\end{definition}

\noindent Some examples of algebraic structures are groups, rings, loops and lattices. We define these examples of algebraic structures.
\\
Groups: A $group$ is an algebra $(G,\cdot, ^{-1}, e)$ with a binary, a unary, and a nullary operation in which the following identities are true:

G1: $x \cdot (y \cdot z) = (x \cdot y) \cdot z$

G2: $x \cdot e = e \cdot x = x$

G3: $x \cdot x^{-1} = x^{-1} \cdot x = e$.

A group is $Abelian$ if the following also holds:

G4: $x \cdot y = y \cdot x$.
\\
Rings: A $ring$ is an algebra $(R,+,\cdot,-,0)$ where $+$ and $\cdot$ are binary, $-$ is unary, and $0$ is nullary, satisfying the following conditions:

R1: $(R,+,-,0)$ is an Abelian group

R2: For $(R,\cdot)$, has the property (G1)

R3: $x \cdot (y+z) = (x \cdot y)+(x \cdot z)$

$(x+y) \cdot z = (x \cdot z) + (y \cdot z)$.
\\
Loops: A $loop$ is an algebra $(Q,\cdot,\backslash,/, e)$ with three binary operations and a nullary that satsify (G2) and the following identities:

Q1: $x \backslash (x \cdot y) = y; (x \cdot y) / y = x$

Q2: $x \cdot (x \backslash y) = y; (x / y) \cdot y = x$.
\\
Lattices: A $lattice$ is an algebra $(L, \vee, \wedge)$ with two binary operations \textit{join} and \textit{meet} satisfying the following conditions:

\begin{tabular}{llll}

L1:& $x \vee y \approx y \vee x$& \hspace{1cm}&

 $x \wedge y \approx y \wedge x$ \\

L2: &$x \vee (y \vee z) \approx (x \vee y) \vee z)$&&

 $x \wedge (y \wedge z) \approx (x \wedge y) \wedge z)$ \\

L3:& $x \vee x \approx x$&&

 $x \wedge x \approx x$ \\

L4:& $x \approx x \vee (x \wedge y)$&&

 $x \approx x \wedge (x \vee y)$.
\end{tabular}

There is an equivalent definition using partial orders. A partial order $L$ is a lattice if and only if for every $x,y$ in $L$ both $sup\{x,y\}$ and $ inf\{x,y\}$ are in $L$ \cite{univ_alg}*{5--8}.
As lattices are partially ordered sets, we can use Hasse diagrams to display them. Examples of Hasse diagrams of lattices can be found in Figure~\ref{lattice_ex}.


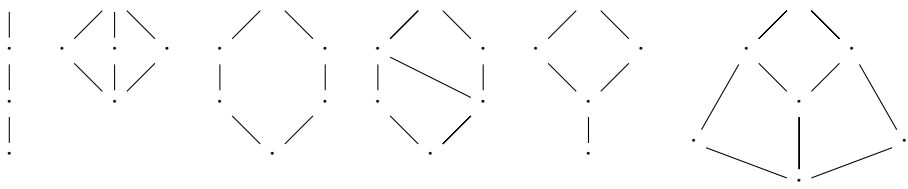
\begin{figure}
    \centering
    \caption{Hasse Diagrams of Lattices}
    \label{lattice_ex}

\begin{tikzpicture}[scale=.7]
  \node (one) at (0,2) {$\cdot$};
  \node (a_1) at (0,1) {$\cdot$};
  \node (a_2) at (0,0) {$\cdot$};
  \node (zero) at (0,-1) {$\cdot$};
  \draw (one) -- (a_1) -- (a_2) -- (zero);
  \node (one2) at (2,2) {$\cdot$};
  \node (a_12) at (1,1) {$\cdot$};
  \node (a_22) at (2,1) {$\cdot$};
  \node (a_32) at (3,1) {$\cdot$};
  \node (zero2) at (2,0) {$\cdot$};
  \draw (one2) -- (a_12) -- (zero2)  -- (a_32) -- (one2) -- (a_22) -- (zero2);
    \node (one3) at (5,2) {$\cdot$};
  \node (a_13) at (4,1) {$\cdot$};
  \node (a_23) at (4,0) {$\cdot$};
  \node (a_33) at (6,1) {$\cdot$};
  \node (a_43) at (6,0) {$\cdot$};
  \node (zero3) at (5,-1) {$\cdot$};
  \draw (one3) -- (a_13) -- (a_23) -- (zero3)  -- (a_43) -- (a_33) -- (one3);
  \node (one4) at (8,2) {$\cdot$};
  \node (a_14) at (7,1) {$\cdot$};
  \node (a_24) at (7,0) {$\cdot$};
  \node (a_34) at (9,1) {$\cdot$};
  \node (a_44) at (9,0) {$\cdot$};
  \node (zero4) at (8,-1) {$\cdot$};
  \draw (one4) -- (a_14) -- (a_24) -- (zero4)  -- (a_44) -- (a_34) -- (one4) -- (a_14) -- (a_44) -- (zero4);
  \node (one5) at (11,2) {$\cdot$};
  \node (a_15) at (10,1) {$\cdot$};
  \node (a_25) at (12,1) {$\cdot$};
  \node (a_35) at (11,0) {$\cdot$};
  \node (zero5) at (11,-1) {$\cdot$};
  \draw (one5) -- (a_15) -- (a_35) -- (zero5)  -- (a_35) -- (a_25) -- (one5);
\node (one6) at (15,2) {$\cdot$};
  \node (a_16) at (14,1) {$\cdot$};
  \node (a_26) at (16,1) {$\cdot$};
  \node (a_36) at (15,0) {$\cdot$};
  \node (a_46) at (13,-.75) {$\cdot$};
  \node (a_56) at (17,-.75) {$\cdot$};
  \node (zero6) at (15,-1.5) {$\cdot$};
  \draw (one6) -- (a_16) -- (a_46) -- (zero6)  -- (a_56) -- (a_26) -- (one6) -- (a_16) -- (a_36) -- (zero6) -- (a_36) -- (a_26) -- (one6);
\end{tikzpicture}

\end{figure}
As a majority of this paper will deal with lattices, we will go over the required terminology.
For lattice $L$, $L'$ is called the dual lattice if there is an isomorphism $f:L$ to $L'$ where for $x,y\in L$, $x\wedge y=f(x) \vee ' f(y)$ and $x\vee y=f(x) \wedge ' f(y)$.
We call two elements $x,y$ $incomparable$ when $x \not\leq y$ and $y\not\leq x$. If a lattice has no incomparable elements then it is a totally or linearly ordered set.
We say $x$ $covers$ $y$ if $x<y$ and, for $y \leq z \leq x$, $z = y$ or $ z = x$. Meaning there are no elements between $x$ and $y$.

Going back to algebraic structures in general, let's look at what it means to be a subalgebra. 
For example, the subalgebra of a group is a subgroup, and the subalgebra of a ring is a subring. Subalgebras for an algebraic structure $\textbf{A}$ form a lattice with the ordering of set inclusion. $\textbf{B}$ is a $subalgebra$ of $\textbf{A}$ if $B \subseteq A$ and every fundamental operation of $\textbf{B}$ is a restriction of a corresponding operation of  $\textbf{A}$ \cite{univ_alg}*{31}. 
We will let $\textbf{Sub}(\textbf{A})$ denote the subalgebra lattice for $\textbf{A}$. Birkhoff and Frink proved in 1948 the following theorem:

\begin{theorem}[Birkhoff and Frink] \cite{birk_frink}
If $\textbf{L}$ is an compactly generated lattice, then $\textbf{L} \cong \textbf{Sub}(\textbf{A})$, for some algebra $\textbf{A}$. 
\end{theorem}

The proof of the Birkhoff and Frink theorem is constructive. Burris and Sankappanavar give a good undergraduate-level proof \cite{univ_alg}*{34}. Following the proof allows one to find the operations that generate an algebraic structure with an isomorphic subalgebra lattice to any given lattice.
All lattices in this paper are finite, which means they are compactly generated.

 As mentioned in the introduction, we will be building algebraic structures given a lattice. We will use the following definition for clarity and brevity in our propositions.

\begin{definition}
An algebraic representation of a lattice $L$ is an algebraic structure whose subalgebra lattice is isomorphic to $L$.
\end{definition}

We proceed with an example of how Birkhoff and Frink's method is used to compute the algebraic representation of $M_2$, the minimal lattice with incomparable elements as shown in Figure~\ref{m_2}.

\begin{figure}
    \centering
    \caption{$M_2$ Lattice}
    \label{m_2}

\begin{tikzpicture} [scale= .7]
  \node (one) at (0,2) {$1$};
  \node (a_1) at (1,1) {$a_1$};
  \node (a_2) at (-1,1) {$a_2$};
  \node (zero) at (0,0) {$0$};
  \draw (one) -- (a_1) -- (zero) -- (a_2)  -- (one);
\end{tikzpicture}

\end{figure}
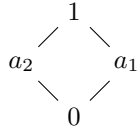


Let $B$ be a subset of $L$, $L= \{1,a_1,a_2,0\}$. Let $\langle B \rangle$ be the closure of $B$, where $\langle B \rangle=\{a\in L$ $ |$ $ a\leq \bigvee_{b\in B} b\}$.

For example, if $B = \{a_1\}$ then the join of $a_1$ is $a_1$, so $\langle \{a_1\} \rangle = \{a_1,0\}$. If $B = \{a_1, a_2\}$ then $a_1 \vee a_2 = 1$. So, $\langle \{a_1, a_2\} \rangle = \{1,a_1,a_2,0\}$. 
The Birkhoff and Frink method creates $n$-ary functions as follow.
Given $B \subseteq L$ 
we create $n-ary$ functions for each 
$b \in \langle B \rangle$,
\[
f_{B,b}(a_1,...,a_n)=
\begin{cases}
b &\text{if } B=\{a_1,...,a_n\},\\
a_1 &\text{ otherwise, } 
\end{cases}
\]
where $a_i \in L$.
Here are the functions for the following sets: $\emptyset, \{a_1\}, \{a_1,a_2\}$

\[
f_{\emptyset,0}=0
\qquad f_{\{a_1,a_2\},0}(x_1,x_2)=
\begin{cases}
0 &\text{if } \{a_1,a_2\}=\{x_1,x_2\},\\
x_1 &\text{otherwise .} 
\end{cases}\]

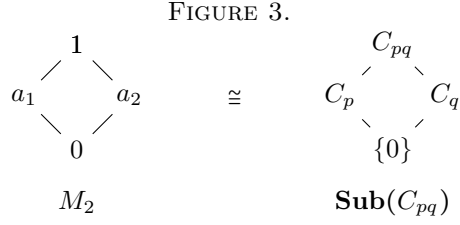
\[
f_{\{a_1\},0}(x)=
\begin{cases}
0 &\text{if } a_1=x,\\
x &\text{otherwise .} 
\end{cases}
\qquad
f_{\{a_1,a_2\},1}(x_1,x_2)=
\begin{cases}
1 &\text{if } \{a_1,a_2\}=\{x_1,x_2\},\\
x_1 &\text{otherwise .} 
\end{cases}
\]
\[
f_{\{a_1\},a_1}(x)=
\begin{cases}
a_1 &\text{if } a_1=x,\\
x &\text{otherwise .} 
\end{cases} \qquad 
f_{\{a_1,a_2\},a_1}(x_1,x_2)=
\begin{cases}
a_1 &\text{if } \{a_1,a_2\}=\{x_1,x_2\},\\
x_1 &\text{otherwise .} 
\end{cases}
\]
\[
f_{\{a_1,a_2\},a_2}(x_1,x_2)=
\begin{cases}
a_2 &\text{if } \{a_1,a_2\}=\{x_1,x_2\},\\
x_1 &\text{otherwise .} 
\end{cases}
\]

To conserve space we will not list out all 50 functions that the Birkhoff and Frink theorem would describe.
We do not need all 50 operations to create the same closed sets. 
We can eliminate any $f_{B,b}(a_1,...,a_n)$ where $b \in B$ since these functions do not give us any of the new elements needed to complete the closure. Secondly, we eliminate $n-ary$ functions where the needed elements are generated from different subsets with $(n-1) -ary$ functions. So for example, we do not need the unary operation $f_{\{a_1\},0}(x)$ to pick up the $0$ element, because we already have the nullary operation $f_{\emptyset,0}$ generating $0$. Another example to consider is the function that picks up the greatest element $1$. By our first rule, there is no such unary function that generates $1$ and there is only one such binary function, $f_{\{a_1,a_2\},1}(\{x_1,x_2\})$.
After we eliminate operations using these rules we are left with $4$ functions: $f_{\emptyset 0}, f_{\{1\},a_1}(x), f_{\{1\},a_2}(x), f_{\{a_1,a_2\},1}(\{x_1,x_2\})$.

From this example we can see that the Birkhoff and Frink method can give us an excess of functions even for lattices with very few elements.

In this case, four functions is still more than necessary. The subalgebra lattice for the cyclic group $C_{pq}$ is isomorphic to $M_2$, see Figure~\ref{m2-group-zpq}. This group has three functions, where as the reduced method gave us four.

In the following sections, we will count the number of functions the reduced method gives us for various families of lattices.

\begin{figure}
    \centering
    \caption{}
    \label{m2-group-zpq}

\begin{tikzpicture} [scale= .7]
  \node (one) at (0,2) {$1$};
  \node (a_1) at (1,1) {$a_2$};
  \node (a_2) at (-1,1) {$a_1$};
  \node (zero) at (0,0) {$0$};
  \node (name) at (0,-1) {$M_2$};
  \draw (one) -- (a_1) -- (zero) -- (a_2)  -- (one);  \node (one) at (0,2) {$1$};
  \node (iso) at (3,1) {$\cong$};
  \node (one2) at (6,2) {$C_{pq}$};
  \node (a_12) at (7,1) {$C_q$};
  \node (a_22) at (5,1) {$C_p$};
  \node (zero2) at (6,0) {$\{0\}$};
  \node (name2) at (6,-1) {$\textbf{Sub}(C_{pq})$};
  \draw (one2) -- (a_12) -- (zero2) -- (a_22)  -- (one2);
\end{tikzpicture}
\end{figure}
\section{Chain Lattices}

The first family of lattice we consider is the simplest, the chain lattice. A chain lattice is a lattice whose set of $n$ elements is a totally ordered set, see Figure~\ref{chain}. 

\begin{figure}
    \centering
    \caption{Chain Lattice}
    \label{chain}

\begin{tikzpicture}[scale=.7]
  \node (one) at (0,2) {$1$};
  \node (a_1) at (0,1) {$a_{n-2}$};
  \node (a_2) at (0,0) {$\vdots$};
  \node (a_4) at (0,-1) {$a_{1}$};
  \node (zero) at (0,-2) {$0$};
  \draw (one) -- (a_1) -- (a_2)  -- (a_4) -- (zero);
\end{tikzpicture}
\end{figure}
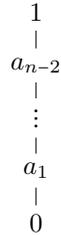

\begin{proposition}[Chain Lattice]\label{ch}
Given a chain lattice $L$ of $n$ elements, there exists an algebraic representation with $n$ elements and $n-1$ functions.
\end{proposition}
\begin{proof}
We will proceed by induction. Let $L$ be a chain lattice. 

If $L$ has $2$ elements, $L = \{0,1\}$, then there is only $1$ function  $f_{\emptyset 0}$. 
If $L$ has $3$ elements, $L = \{0,a,1\}$, then we have two functions $f_{\emptyset 0}$ and $f_{\{1\},a}(x)$.
If $L$ has $4$ elements, $L = \{0,a_1,a_2,1\}$, then we have $3$ functions $f_{\emptyset 0}$, $f_{\{a_2\},a_1}(x)$ and $f_{\{1\},a_2}(x)$.

Now, assume for $L$ with $n$ elements we need $n-1$ functions. Let $L^*$ be a chain of $n+1$ elements where $1^*$ is the greatest element and all other elements remain named the same as in $L$.
The closures we need to consider are $\langle 0 \rangle,\langle a_1 \rangle, \langle a_2 \rangle, ..., \langle a_{n-2} \rangle, \langle 1 \rangle$, $\langle1^*\rangle$. The functions given by $L$ create the closures for $\langle 0 \rangle$ to $\langle 1 \rangle$. So we need functions to cover $\langle 1^* \rangle$. 
\[
f_{\{1^*\},1}(x)=
\begin{cases}
1 &\text{if } 1^*=x,\\
x &\text{ otherwise } 
\end{cases}
\]
The $\langle 1^* \rangle = \langle 1 \rangle \cup \{1^*\}$. The functions generated by $L$ pick up the elements of $\langle 1 \rangle$. So we only need one additional function $f_{\{1^*\},1}(x)$ to cover $\langle 1^* \rangle$. Thus $L^*$ has $n+1$ elements and $n$ functions.

\end{proof}

There are known algebraic representations for the chain lattice. One such algebraic structure is groups, more precisely the cyclic groups $C_{p^{n-1}}$.

\section{Variations on the Chain Lattice}

From here we will consider lattices that are variations of the chain lattice. We will start with the simplest variation on the chain, the $N_k$ lattice. A $N_k$ lattice is a lattice with $k+3$ elements made up of a least element $0$, a greatest element $1$, and a chain of $k$ elements $0<b_1<...<b_k<1$. There is also an element $a$ which is incomparable to all $b_i$ $ i \in \{1,...,k\}$.
Figure~\ref{n-k} shows the Hasse diagram of the $N_k$ lattice.

\begin{figure}
    \centering
    \caption{$N_k$ Lattice}
    \label{n-k}

\begin{tikzpicture}[scale=.7]
  \node (one) at (0,2) {$1$};
  \node (a) at (-1.5,0) {$a$};
  \node (b_1) at (1.5,.75) {$b_1$};
  \node (b_2) at (1.5,0) {$\vdots$};
  \node (b_3) at (1.5,-.75) {$b_k$};
  \node (zero) at (0,-2) {$0$};
  \draw (b_1) -- (one) -- (a) -- (zero) -- (b_3);
\end{tikzpicture}
\end{figure}
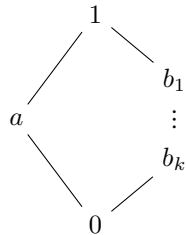

Visually we can see that the $N_k$ appears to be very similar to a chain of $k+2$ elements. The algebraic representation of $N_k$ will have many more functions than the chain.

\begin{proposition}[$N_k$ Lattice]\label{nk}
Given an $N_k$ lattice, there exists an algebraic representation of $k+3$ elements with $2k+2$ functions.
\end{proposition}

\begin{proof}
To build up the operations of the algebra, first consider the chain consisting of $0$, $b_1, b_2,...,b_k,$ and $1$. We can apply the same functions obtained from Proposition~\ref{ch} to this chain, giving us $k+1$ functions. Note that this includes the nullary operation. We also will create a function that gives us element $a$ from  element $1$ as shown below.
\[
f_{\{1\},a}(x)=
\begin{cases}
a &\text{if } x=1,\\
x &\text{ otherwise } 
\end{cases}
\]
Lastly, for each $b_i$ where $1\leq i\leq k$, we will create a binary function that produces element $1$ from $\{b_i, a\}$.
\[
f_{\{b_i, a\},1}(x_1, x_2)=
\begin{cases}
1 &\text{if } \{x_1,x_2\}=\{b_i,a\},\\
x_1 &\text{ otherwise } 
\end{cases}
\]
In total, this is $k+1+1+k=2k+2$ functions.
\end{proof}

For an $N_k$ lattice, there does not exist a group representation, with the exception of $k=1$ \cite{subgroup}. The $N_1$ is isomorphic to $M_2$ lattice, the example we used to examine the Birkhoff and Frink method.

Now we will consider a variation on the $N_k$ lattice, what we have named the $N_k^*$ lattice. A $N^*_k$ lattice is a lattice with $k+4$ elements. It is made up of a least element $0$, a greatest element $1$, and a chain of $k$ elements $1<b_1,...,b_k<0$. Unlike the $N_k$ lattice, the $N_k^*$ has two elements $a_1, a_2$ which are incomparable to each other and all $b_i$ for $ i \in \{1,...,k\}$.
Figure~\ref{n-k*} shows the Hasse diagram of $N_k^*$.

\begin{figure}
    \centering
    \caption{$N_k*$ Lattice}
    \label{n-k*}

\begin{tikzpicture}[scale=.7]
  \node (one) at (0,2) {$1$};
  \node (a_1) at (-1.5,0) {$a_1$};
  \node (a_2) at (0,0) {$a_2$};
  \node (b_1) at (1.5,.75) {$b_1$};
  \node (b_2) at (1.5,0) {$\vdots$};
  \node (b_3) at (1.5,-.75) {$b_k$};
  \node (zero) at (0,-2) {$0$};
  \draw (b_1) -- (one) -- (a_1) -- (zero) -- (b_3) -- (zero) -- (a_2) -- (one);
\end{tikzpicture}
\end{figure}
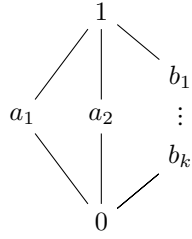

Just as the $N_k$ is a chain with one additional incomparable element, the $N_k^*$ is a $N_k$ lattice with another additional incomparable element. Now we will examine how this change in structure effects the algebra that generates our subalgebra lattice $N_k^*$.

\begin{proposition}[$N^*_k$ Lattice]\label{nk*}
Given a $N_k^*$ lattice, there exists an algebraic representation with $k+4$ elements and $3k+4$ functions.
\end{proposition}
\begin{proof}
Consider $N_k^*$. Note that the lattice $N_k$ is embedded in $N_k^*$. By applying Proposition~\ref{nk} for $N_k$, we get $2k+2$ functions, assume that this includes $a_1$ as our element $a$ in $N_k$ and excludes the element $a_2$. Next we need the function that gives us $a_2$ from $1$, $f_{\{1\},a_2}(x)$. Lastly we need $k+1$ functions pairing $a_2$ with $a_1,b_1,...,b_k$ to generate $1$. In total there are $2k+2+1+k+1 = 3k+4$ functions.
\end{proof}

In general a $N_k^*$ is not a subalgebra lattice of a group. For $k=1$ however, $N_1^*$ is the subalgebra lattice of the Klein $4$ group. For $n\geq2$ the algebras are unknown.

The next variation on the chain we consider is the $M_n$ lattice. An $M_n$ lattice is a lattice of order $n$ incomparable elements $a_1,...,a_n$ with $0$ and $1$. 

\begin{figure}
    \centering
    \caption{$M_n$ Lattice}
    \label{m-n}

\begin{tikzpicture}[scale=.7]
  \node (one) at (0,2) {$1$};
  \node (a_1) at (2,0) {$a_n$};
  \node (a_2) at (-2,0) {$a_1$};
  \node (a_3) at (1, 0) {...};
  \node (a_4) at (-1,0) {...};
  \node (a_5) at (0,0) {...};
  \node (zero) at (0,-2) {$0$};
  \draw (one) -- (a_1) -- (zero) -- (a_2) -- (one) -- (a_3) -- (zero) -- (a_4) -- (one) -- (a_5) -- (zero);
\end{tikzpicture}
\end{figure}
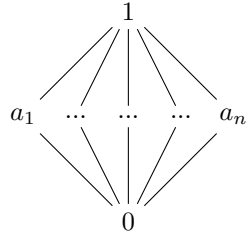

Examining Figure~\ref{m-n}, we can see that the $M_n$ lattice is a collection of $n$ chains sharing a least and greatest element. We now examine how this change in structure effects the algebra that generates our subalgebra lattice $M_n$.

\begin{proposition}[$M_n$ Lattice]
Given a $M_n$ lattice, there exists an algebraic representation with $n+2$ elements and $\frac{n(n+1)}{2} +1$ functions.
\end{proposition}
\begin{proof}
As before we will need $\emptyset$ and $\langle 1 \rangle$. The $\emptyset$ again gives us the nullary function $f_{\emptyset 0}$. The $\langle 1 \rangle$ gives $n$ functions $f_{\{1\}, a_i}$. For each $\langle a_i,a_j \rangle$ where $1\leq i, j \leq n$ and $i\neq j$, $\langle a_i,a_j \rangle$ gives one function $f_{\{a_i,a_j\}, 1}$. Together there are ${n \choose 2} = \frac {n(n-1)}{2}$ functions $f_{\{a_i, a_j\}, 1}$. All together there are
$1+n+ \frac{n(n-1)}{2}=\frac{n(n+1)}{2} +1$.
Thus, the lattice $M_n$ has $\frac{n(n+1)}{2} +1$ functions.
\end{proof}

For $n=1,2,3$ we have already examined the algebras whose subalgebra lattices are isomorphic to $M_n$. Where $n>3$, it is known that none of the algebraic representations of $M_n$ are groups.  For $n= p^q$ where $p$ is prime and $q\geq0$, Foguel and Hiller showed there is  a loop whose subalgebra lattice is isomorphic to $M_n$ \cite{loop}*{14}.

\section{Pinecone and Christmas Tree Lattices}


The next type of lattices we wanted to consider are those where special elements build the rest of the lattice. Pinecone lattices are generated by elements covered by one, the maximal elements. Christmas Tree lattices are generated by the elements that cover 0, the minimal elements.
\begin{definition}
A Pinecone lattice $P_n$ is defined as follows:

P1: $P_n$ has $n$ maximal elements $a_1,...,a_n$

P2: For $i\neq j$ and $1 \leq i,j \leq n$, $a_i \vee a_j = 1$ 

P3: For $b \in P_n/\{1\}$, $a_i \wedge a_j = b$ for exactly one pair of $i$ and $j$

P4: For $i \neq k$ and $1 \leq k \leq n$, $a_i \wedge a_j \neq a_k \wedge a_j$ and $a_i \wedge a_k = a_i \wedge a_{i+1} \wedge ... \wedge a_k$.

\end{definition}

\begin{proposition}[Pinecone Lattice]
Given a Pinecone lattice $P_n$ with $n$ maximal elements, there exists an algebraic representation for $P_n$ with $\frac{n^2+n+2}{2}$ elements and $\frac{3n^2-n-2}{2}$ functions.
\end{proposition}

\begin{proof}
We will proceed by induction. We begin with the base case of $n=2$, see Figure~\ref{p_2}.
\begin{figure}
    \centering
    \caption{$P_2$ Lattice}
    \label{p_2}
\begin{tikzpicture} [scale= .7]
  \node (one) at (0,2) {$1$};
  \node (a_1) at (1,1) {$a_1$};
  \node (a_2) at (-1,1) {$a_2$};
  \node (zero) at (0,0) {$0$};
  \node (name) at (0,-1) {$P_2$};
  \draw (one) -- (a_1) -- (zero) -- (a_2)  -- (one);
\end{tikzpicture}
\end{figure}
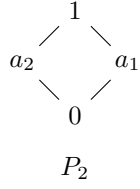

This lattice is isomorphic to the $N_1$, $M_2$, and $N_0^*$. From all the previous propositions regarding those lattices, we see that we can create an algebraic representation of $4$ elements and $4$ operations, as before.


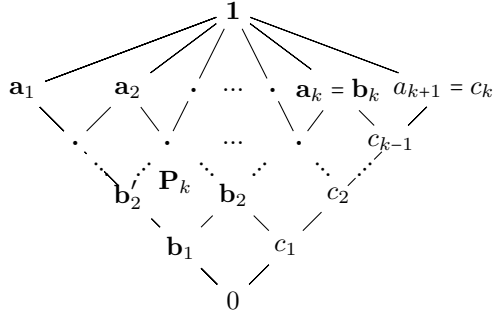
\begin{figure}
    \centering
    \caption{Pinecone Lattice}
    \label{pinecone}
\begin{tikzpicture}[scale=.7]
  \node (one) at (0,0) {${0}$};
  \node (a_1) at (-1,1) {$\textbf{b}_1$};
  \node (a_2) at (1,1) {$c_1$};
  \node (a_3) at (0,2) {$\textbf{b}_2$};
  \node (a_4) at (-2.5,2.5) {$\ddots$};
  \node (a_4*) at (-2,2) {$\textbf{b}'_2$};
  \node (a_5) at (2.5,2.5) {$\udots$};
  \node (a_5*) at (2,2) {$c_2$};
  \node (a_6) at (3,3) {$c_{k-1}$};
  \node (a_7) at (-3,3) {$\textbf{.}$};
  \node (a_8) at (2,4) {$\textbf{a}_k = \textbf{b}_k$};
  \node (a_9) at (-2,4) {$\textbf{a}_2$};
  \node (a_10) at (.75,4) {$\textbf{.}$};
  \node (a_11) at (1.25,3) {$\textbf{.}$};
  \node (a_12) at (-.75,4) {$\textbf{.}$};
  \node (a_13) at (-1.25,3) {$\textbf{.}$};
  \node (a_{n-1}) at (-4,4) {$\textbf{a}_1$};
  \node (a_n) at (4,4) {$a_{k+1} = c_{k}$};
  \node (zero) at (0,5.5) {$\textbf{1}$};
  \node (p_k) at (-1.1,2.25) {$\textbf{P}_k$};
  \node (dots) at (.5,2.5) {$\textbf{$\udots$}$};
  \node (ddots) at (-.5,2.5) {$\ddots$};
  \node (dotss) at (0,3) {...};
  \node (dotts) at (0,4) {...};
  \node (doots) at (1.7,2.5) {$\ddots$};
  \node (dts) at (-1.7,2.5) {$\udots$};
  \draw (one) -- (a_1) -- (a_4*) -- (a_4) -- (a_7) -- (a_{n-1}) -- (zero)  -- (a_n) -- (a_6) -- (a_5) -- (a_5*) -- (a_2) -- (one) -- (a_1) -- (a_3) -- (a_2) -- (one) -- (a_1) -- (a_4*) -- (a_4) -- (a_7) -- (a_{n-1}) -- (zero) -- (a_8) -- (a_6) -- (a_n) -- (zero) -- (a_9) -- (a_7) -- (a_{n-1}) -- (zero) -- (a_10) -- (a_11) -- (a_8) -- (zero) -- (a_9) -- (a_13) -- (a_12) -- (zero);
\end{tikzpicture}
\end{figure}

Assume that for a Pinecone lattice with $k$ maximal elements, the statement holds.  
Let us consider the Pinecone lattice with $k+1$ maximal elements. For every element of the lattice, we will create an element of the algebra.
We will label $P_{k+1}$ as shown in Figure~\ref{pinecone} with $b_1$ and $c_1$ being the minimal elements of the lattice. Notice that the set $\uparrow b_1=\{x \in P_{k+1}:x\geq b_1\}$ is isomorphic to $P_{k}$. We will apply the operations of $P_k$ to $\uparrow b_1$ as if it were $P_k$. We will change the nullary operation and add operations to include the elements not in the upset of $b_1$. The nullary operation from $P_k$ is $f_{\emptyset, b_1}=b_1$. We will change this to $f_{\emptyset, 0}=0$. With the shift of the nullary operation, we will also need the following two functions:
\[
f_{\{b_2\},b_1}(x)=
\begin{cases}
b_1 &\text{if } x=b_2,\\
x &\text{ otherwise } 
\end{cases}
\text{ and }
f_{\{b_2'\},b_1}(x)=
\begin{cases}
b_1 &\text{if } x=b_1',\\
x &\text{ otherwise } 
\end{cases}
\]

where $b_2'$ and $b_2$ are the elements that cover $b_1$.
Now we will consider the elements of $P_{k+1}$ not in $\uparrow b_1$. For example, much like the chain, if you have $c_{i+1}$, you want to be sure you generate $c_{i}$. If you have $c_{i}$ and $b_{i}$, you want to be able to generate $b_{i+1}$. To do this, we create functions of the form
\[
f_{\{b_i, c_i\},b_{i+1}}(x_1,x_2)=
\begin{cases}
b_{i+1} &\text{if } \{x_1,x_2\}=\{b_i,c_i\},\\
x_1 &\text{ otherwise } 
\end{cases}
\]

\[
f_{\{b_{i+1}\},c_i}(x)=
\begin{cases}
c_i &\text{if } x=b_{i+1},\\
x &\text{ otherwise } 
\end{cases}
\]

\[
f_{\{c_{i+1}\},c_i}(x)=
\begin{cases}
c_{i+1} &\text{if } x=c_i,\\
x &\text{ otherwise } 
\end{cases}
\]

where $1\geq i\geq k-1$. 
We also create the following functions where element $1$ acts as $c_{i+1}$ in the previous functions.

\[
f_{\{b_k, c_k\},1}(x_1,x_2)=
\begin{cases}
1 &\text{if } \{x_1,x_2\}=\{b_k,c_k\},\\
x_1 &\text{ otherwise } 
\end{cases}
\]

\[
f_{\{1\},c_k}(x)=
\begin{cases}
c_k &\text{if } x=1,\\
x &\text{ otherwise } 
\end{cases}
\]

This completes our representation and gives us a total of
\\
$\frac{3k^2-k+2}{2}+2+3(k-1)+2=\frac{3k^2-k+2}{2}+3k+1=\frac{3(k+1)^2-(k+1)+2}{2}$
functions. 
\end{proof}

For $n>2$, $P_n$ does not have a group representation. However, $P_3$ is known to be isomorphic to a subloop lattice \cite{pn_loop}. For $n>3$, it is unknown whether $P_n$ has a loop representation or not. 

\begin{definition}
A Christmas Tree lattice, $T_n$, is defined as follows:

T1: $T_n$ has $n$ minimal elements $a_1,...,a_n$

T2: For $i\neq j$ and $1 \leq i, j \leq n$, $a_i \wedge a_j = 0$

T3: For $b \in T_n/\{0\}$, $a_i \vee a_j = b$ for exactly one pair of $i$ and $j$

T4: For $i \neq k$ and $1 \leq k \leq n$, $a_i \vee a_j \neq a_k \vee a_j$ and $a_i \vee a_k = a_i \vee a_{i+1} \vee ... \vee a_k$
\end{definition}.

The Christmas Tree lattice is the dual of the Pinecone lattice. A $T_n$ lattice has the same number of elements as the $P_n$. Many of the functions will be similar; however, their representations do have different numbers of functions.

\begin{proposition}[Christmas Tree Lattice]
Given a Christmas Tree lattice with $n$ minimal elements, there exists an algebraic representation with $\frac{n^2+n+2}{2}$ elements and $\frac{3n^2-3n+2}{2}$ functions.
\end{proposition}
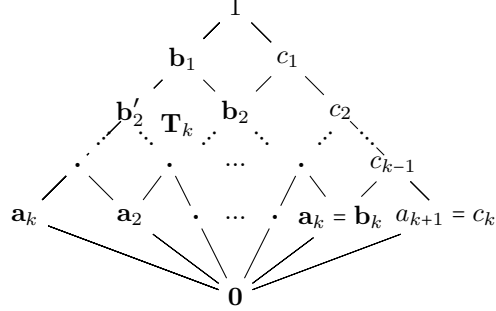
\begin{figure}
    \centering
    \caption{Christmas Tree Lattice}
    \label{christmas}
    \begin{tikzpicture}[scale=.7]
  \node (one) at (0,0) {${1}$};
  \node (a_1) at (-1,-1) {$\textbf{b}_1$};
  \node (a_2) at (1,-1) {$c_1$};
  \node (a_3) at (0,-2) {$\textbf{b}_2$};
\node (a_4*) at (-2,-2) {$\textbf{b}'_2$};
  \node (a_4) at (-2.5,-2.5) {$\udots$};
\node (a_5*) at (2,-2) {$c_2$};
  \node (a_5) at (2.5,-2.5) {$\ddots$};
  \node (a_6) at (3,-3) {$c_{k-1}$};
  \node (a_7) at (-3,-3) {$\textbf{.}$};
  \node (a_8) at (2,-4) {$\textbf{a}_k = \textbf{b}_k$};
  \node (a_9) at (-2,-4) {$\textbf{a}_2$};
  \node (a_10) at (.75,-4) {$\textbf{.}$};
  \node (a_11) at (1.25,-3) {$\textbf{.}$};
  \node (a_12) at (-.75,-4) {$\textbf{.}$};
  \node (a_13) at (-1.25,-3) {$\textbf{.}$};
  \node (a_{n-1}) at (-4,-4) {$\textbf{a}_k$};
  \node (a_n) at (4,-4) {$a_{k+1} = c_k$};
  \node (zero) at (0,-5.5) {$\textbf{0}$};
  \node (t_k) at (-1.1,-2.25) {$\textbf{T}_k$};
\node (dots) at (.5,-2.5) {$\ddots$};
  \node (ddots) at (-.5,-2.5) {$\udots$};
  \node (dotss) at (0,-3) {...};
  \node (dotts) at (0,-4) {...};
\node (doots) at (1.7,-2.5) {$\udots$};
  \node (dts) at (-1.7,-2.5) {$\ddots$};
  \draw (one) -- (a_1) -- (a_4*) -- (a_4) -- (a_7) -- (a_{n-1}) -- (zero)  -- (a_n) -- (a_6) -- (a_5) -- (a_5*) -- (a_2) -- (one) -- (a_1) -- (a_3) -- (a_2) -- (one) -- (a_1) -- (a_4*) -- (a_4) -- (a_7) -- (a_{n-1}) -- (zero) -- (a_8) -- (a_6) -- (a_n) -- (zero) -- (a_9) -- (a_7) -- (a_{n-1}) -- (zero) -- (a_10) -- (a_11) -- (a_8) -- (zero) -- (a_9) -- (a_13) -- (a_12) -- (zero);
\end{tikzpicture}
\end{figure}

\begin{proof}
We begin by considering the base case of $n=2$.
Note, $P_2$ and $T_2$ are isomorphic, meaning that $T_2$ is also isomorphic to $N_1, M_2, N_0^*,$ and $P_2$.
This lattice will generate 4 operations. Plugging $n=2$ into our equation, we get $\frac{3(2)^2-3(2)+2}{2}=4$.

Assume that for a Christmas Tree lattice of $k$ minimal elements, the proposition holds.
Let us consider the lattice with $k+1$ minimal elements. We will label $T_{k+1}$ as noted in Figure~$\ref{christmas}$. 
For every element of the lattice, we will create an element of the algebra. Notice that the set $\downarrow b_1=\{x \in T_{k+1}:x\leq b_1\}$ is isomorphic to $T_{k}$. We will apply the operations of $T_k$ to $\downarrow b_1$ as if it were $T_k$. Since the zero element of $\downarrow b_1$ is the same as the zero element of $T_k$, we don't have to alter the nullary operation as we did in the proof of the Pinecone lattice. However, there are elements outside of $\downarrow b_1$ for which operations are also needed. We will create functions of the form
\[
f_{\{b_{i+1}, c_{i+1}\},c_i}(x_1,x_2)=
\begin{cases}
c_i &\text{if } \{x_1,x_2\}=\{b_{i+1},c_{i+1}\},\\
x_1 &\text{ otherwise } 
\end{cases}
\]

\[
f_{\{c_i\},b_{i+1}}(x)=
\begin{cases}
b_{i+1} &\text{if } x=c_i,\\
x &\text{ otherwise } 
\end{cases}
\]

\[
f_{\{c_i\},c_{i+1}}(x)=
\begin{cases}
c_i &\text{if } x=c_{i+1},\\
x &\text{ otherwise } 
\end{cases}
\]

where $1\geq i\geq k-1$. 
We will also create the following functions:

\[
f_{\{b_1, c_1\},1}(x_1,x_2)=
\begin{cases}
1 &\text{if } \{x_1,x_2\}=\{b_1,c_1\},\\
x_1 &\text{ otherwise } 
\end{cases}
\]

\[
f_{\{1\},c_1}(x)=
\begin{cases}
c_1 &\text{if } x=1,\\
x &\text{ otherwise } 
\end{cases}
\]

\[
f_{\{1\},b_1}(x)=
\begin{cases}
b_1 &\text{if } x=1,\\
x &\text{ otherwise } 
\end{cases}
\]

This completes our representations and gives us a total of \\$\frac{3k^2-3k+2}{2}+3(k-1)+3=\frac{3k^2-3k+2}{2}+3k=\frac{3k^2+3k+2}{2}=\frac{3(k+1)^2-3(k+1)+2}{2}$ functions. 

\end{proof}

\section{Power Set Lattice}

In terms of constructing Birkhoff and Frink's algebraic representations, one of the more difficult families of lattices examined in this paper is the power set lattice. The power set lattice $\mathcal{P}(S)$ is the lattice formed when the partial order of set inclusion is applied to the elements of the power set of a set with $n$ elements. Figure~\ref{powerset-3} is a Hasse diagram of $\mathcal{P}(S)$ where $|S|=3$.

\begin{figure}
    \centering
    \caption{$\mathcal{P}(S)$ where $S = \{1,2,3\}$}
    \label{powerset-3}

\begin{tikzpicture}[scale=.7]
  \node (one) at (0,2.5) {$\{1, 2, 3\}$};
  \node (a_1) at (-2,1) {$\{1, 2\}$};
  \node (a_2) at (0,1) {$\{1, 3\}$};
  \node (a_3) at (2,1) {$\{2, 3\}$};
  \node (a_4) at (-2,-.5) {$\{1\}$};
  \node (a_5) at (0,-.5) {$\{2\}$};
  \node (a_6) at (2,-.5) {$\{3\}$};
  \node (zero) at (0,-2) {$\emptyset$};
  \draw (one) -- (a_1) -- (a_4) -- (zero) -- (a_5) -- (a_1) -- (one) -- (a_2) -- (a_4) -- (zero) -- (a_6) -- (a_2) -- (one) -- (a_3) -- (a_6) -- (zero) -- (a_5) -- (a_3) -- (one);
\end{tikzpicture}
\end{figure}
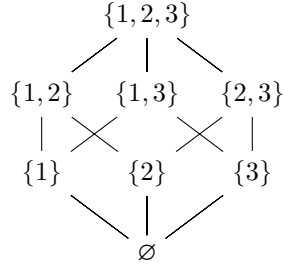

One may recognize the $\mathcal{P}(S)$ lattice where $|S|=3$ as a subgroup lattice. Every power set lattice is isomorphic to a subgroup lattice. The $\mathcal{P}(S)$ lattice is isomorphic to the subgroup lattice of $C_{p_1,...,p_n}$ for distinct primes $p_1,...,p_n$ where $|S|=n$. Thus each $\mathcal{P}(S)$ has an algebraic representation with 3 functions. However, we will consider the Birkhoff and Frink method for finding algebraic representations in the following proposition.

\begin{proposition}[Power Set Lattice]
Given a power set lattice $\mathcal{P}(S)$ where $|S|=n$, there exists an algebraic representation with $2^n$ elements and $2^n - n + \sum_{i=2}^{n} i\binom{n}{i}$ functions.

\end{proposition}

\begin{proof}
Before we start counting functions, we will clear up some notation. Here, $X$ is an element of our lattice $\mathcal{P}(S)$, meaning 
\[
X \in \mathcal{P}(S) \text{ and } X \subseteq S
\]
For $x_i$ an element of $S$ we have $\{x_i\}$ the singleton set in $\mathcal{P}(S)$. These singleton elements $\{x_i\}$ are minimal elements of $\mathcal{P}(S)$.

We begin, as always, with the nullary operation. We then consider every element with $|X|=k > 1$, taking into consideration that each element of the lattice is a set. 

For each of these elements $X$, we will create functions of two forms:
\[
f_{\{X\},\{x_i\}}(A)=
\begin{cases}
\{x_i\} &\text{if } A=X,\\
A &\text{otherwise } 
\end{cases}
\]
\[
f_{\{\{x_1\},\{x_2\}, ...,\{x_k\}\}, X}(A_1, A_2,...,A_k)=
\begin{cases}
X &\text{if } \{A_1, A_2,...,A_k\}\\
&\hspace{1cm}=\{\{x_1\},\{x_2\}, ...,\{x_k\}\}\\
A_1 &\text{otherwise } 
\end{cases}
\]

where $x_i\in X$. The first function will take the element $X$ and give us back the minimal elements less than $X$. These minimal elements will be the singleton subsets of $X$. Thus for each $X\in \mathcal{P}(S)$ with $|X|=k>1$, we will have $k$ of these functions. Note that there are $\sum_{i=2}^{n} \binom{n}{i}=2^n - n - 1$ elements with a carnality greater than one. Thus, overall we will need $\sum_{i=2}^{n} i\binom{n}{i}$ functions of this form.

The second function allows the minimal elements less than $X$ to generate $X$. We need just one of these for each element with carnality greater than one, or $2^n -n-1$ overall.

Adding the number of operations created, we get $1+2^n -n-1 + \sum_{i=2}^{n} i\binom{n}{i}=2^n - n + \sum_{i=2}^{n} i\binom{n}{i}$ total operations.

\end{proof}

\pagebreak

\section{Conclusion}

We considered a few families of lattices.  Each family considered is a very specific type of lattice.  A further look at this topic would include broader families of lattices.  For example, Christmas Tree and power set lattices both have the property that each element is the join of minimal elements. One might consider the family of lattices with this property.  

Further lines of research would also include finding an algebraic representation type for a given family of lattices. In the representations presented, the number of functions changes as the number of elements in the lattice changes. In the case of finite chain lattices and finite power set lattices, each can be represented as a group.  In the group representations, the number of elements in the group would change as the size of lattice changes, but the number of functions stays the same. Of the families of lattices considered in this paper, these are the only two which are known to have the same type of algebraic representation for every lattice in the family.  There is ongoing research on whether or not every lattice can be represented as a loop \cite{loop}.

\section*{References}

\bibliographystyle{amsalpha}

\end{document}